\documentclass{article}
\usepackage[utf8]{inputenc}
\usepackage[babel]{microtype}
\usepackage[english]{babel}

\usepackage{subcaption}

\usepackage{amsmath,amssymb,amsthm, mathrsfs, mathtools,bbm,bm}
\usepackage[usenames,dvipsnames]{xcolor}
\usepackage{multirow}

\usepackage{enumerate}
\usepackage{enumitem}

\setlist[enumerate,1]{label={\arabic*.}}
\setlist[enumerate,2]{label={\roman*.}}
\setlist[enumerate,3]{label={\alph*.}}

\setlist[itemize]{nolistsep,noitemsep, topsep=0pt}
\setlist[enumerate]{nolistsep,noitemsep, topsep=0pt}

\usepackage[hypcap=false]{caption}
\usepackage[bmargin=15mm,tmargin=15mm,lmargin=20mm,rmargin=20mm]{geometry}
\usepackage[hypertexnames=false]{hyperref}
\hypersetup{
    colorlinks,
    linkcolor={red!60!black},
    citecolor={green!60!black},
    urlcolor={blue!60!black}
}

\newcommand{\cL}{\ensuremath{\mathcal L}}

\usepackage[capitalize]{cleveref}
\crefname{equation}{}{}
\crefname{subfigure}{Figure}{Figures}
\crefname{figure}{Figure}{Figures}
\crefname{question}{Question}{Questions}

\theoremstyle{plain}
\newtheorem{theorem}{Theorem}[section]

\newtheorem{proposition}[theorem]{Proposition}
\newtheorem{lemma}[theorem]{Lemma}
\newtheorem{corollary}[theorem]{Corollary}

\theoremstyle{definition}
\newtheorem{definition}[theorem]{Definition}

\newtheorem{remark}[theorem]{Remark}
\newtheorem{question}[theorem]{Question}

\DeclarePairedDelimiter{\parens}{(}{)}
\DeclarePairedDelimiter{\set}{\{}{\}}

\DeclarePairedDelimiter{\card}{\lvert}{\rvert}
\DeclarePairedDelimiter{\size}{\lvert}{\rvert}
\DeclarePairedDelimiter{\floor}{\lfloor}{\rfloor}
\DeclarePairedDelimiter{\ceil}{\lceil}{\rceil}

\newcommand{\ro}[2]{\ensuremath{r_{{\text odd}}(#1,#2)}}
\newcommand{\rod}[3]{\ensuremath{r_{{\text odd}}\parens*{#1,#2;#3}}}

\renewcommand{\phi}{\varphi}

\newcommand{\Erdos}{Erd\H{o}s}
\newcommand{\Gyarfas}{Gy\'arf\'as}

\title{Odd-Ramsey numbers of Hamilton cycles}

\usepackage{graphicx}

\usepackage{ifthen}
\usepackage{tikz}
\usepackage{calc}
\usetikzlibrary{calc,backgrounds}

\tikzstyle{vertex}=[circle, draw, fill=black, inner sep=0pt, minimum size=4pt]
\newcommand{\vertex}{\node[vertex]}

\tikzstyle{smallvertex}=[circle, draw, fill=black, inner sep=0pt, minimum size=3pt]
\newcommand{\smallvertex}{\node[smallvertex]}

\newcommand{\widthedge}{2}

\renewcommand{\widthedge}{1.5}
\newcommand{\switch}[2]{
\begin{tikzpicture}[x=#1 cm, y=#1 cm]

    \vertex (u) at (180:1) [label=180:$u$]{};
    \vertex (v) at (90:1) [label=90:$v$]{};
    \vertex (x) at (0:1) [label=0:$x$]{};
    \vertex (y) at (-90:1) [label=-90:$y$]{};

    \path (v) edge[out=-110, in=80, looseness=2] node[right]{$P$} (y);

    \ifthenelse{#2=0}{
        \path
        (u) edge[blue,line width=\widthedge] (v)
        (v) edge[green,line width=\widthedge] (x)
        (x) edge[red,line width=\widthedge] (y)
        (y) edge[green,line width=\widthedge] (u);
    
    }{}

    \ifthenelse{#2=1}{
        \path
        (u) edge[blue,line width=\widthedge] (v)
        (x) edge[red,line width=\widthedge] (y);

    }{}    

    \ifthenelse{#2=2}{
        \path
        (v) edge[green,line width=\widthedge] (x)
        (y) edge[green,line width=\widthedge] (u);      
    }{}

\end{tikzpicture}
}

\tikzset{
    partial ellipse/.style args={#1:#2:#3}{
        insert path={+ (#1:#3) arc (#1:#2:#3)}
    }
}
        
\renewcommand{\widthedge}{1.5}
\newcommand{\inception}[2]{
\begin{tikzpicture}[x=#1 cm, y=#1 cm]

    \ifthenelse{#2<1}{
        \pgfmathsetmacro\radius{ #1*2}
        \draw (0,0) [partial ellipse=70:-40:\radius cm and \radius cm];
        \draw (0,0) [partial ellipse=-40:-140:\radius cm and \radius cm];
        \draw (0,0) [partial ellipse=-140:-250:\radius cm and \radius cm];
        
        \vertex (b0) at (70:\radius cm) [label=90:$b_1$]{};
        \vertex (a0) at (110:\radius cm) [label=90:$a_1$]{};
        \path (b0) edge[line width=\widthedge] node[below]{$e_1$} (a0);
    }{
        \pgfmathsetmacro\size{ #1*1.8}
        \ifthenelse{#2<2}{
        \vertex (u1) at (180:\size cm) [label=180:$u_1$]{};
        \vertex (v1) at (90:\size cm) [label=90:$v_1$]{};
        \vertex (x1) at (0:\size cm) [label=0:$x_1$]{};
        \vertex (y1) at (-90:\size cm) [label=-90:$y_1$]{};}
        {\vertex (u1) at (180:\size cm) [label=180:]{};
        \vertex (v1) at (90:\size cm) [label=90:]{};
        \vertex (x1) at (0:\size cm) [label=0:]{};
        \vertex (y1) at (-90:\size cm) [label=-90:]{};}

        \pgfmathsetmacro\size{ #1*2}
        \ifthenelse{#2<2}{
        \vertex (b1) at (-30:\size cm) [label=0:$b_1$]{};
        \vertex (a1) at (210:\size cm) [label=180:$a_1$]{};}{
        \vertex (b1) at (-30:\size cm) [label=0:]{};
        \vertex (a1) at (210:\size cm) [label=180:]{};}
        
        \ifthenelse{#2>1}{
        \path 
        (u1) edge[blue,line width=1,dotted] (y1)
        (v1) edge[blue,line width=\widthedge] (u1)
        (x1) edge[blue,line width=1,dotted] (v1)
        (y1) edge[blue,line width=\widthedge] (x1);
         }{
        \path
        (u1) edge[blue,line width=1,dotted] (y1)
        (v1) edge[blue,line width=\widthedge] (u1)
        (x1) edge[blue,line width=1,dotted] (v1)
        (y1) edge[red,line width=\widthedge] (x1);}

        \path
        (u1) edge[orange, line width=\widthedge] (a1)
        (x1) edge[green, line width=\widthedge] (b1);
        
        \path (a1) edge[out=-80, in=-100, looseness=2] (b1);

        \pgfmathsetmacro\size{ #1*0.5}
        
        \ifthenelse{#2<3}{
        \path (v1) edge[line width=\widthedge] node[right]{$e_1$} (y1);}{}

        \ifthenelse{#2>2}{
        \pgfmathsetmacro\size{ #1*.75}
        \smallvertex (u2) at (90:\size cm) [label=95:]{};
        \smallvertex (v2) at (0:\size cm) [label=0:]{};
        \smallvertex (x2) at (-90:\size cm) [label=-85:]{};
        \smallvertex (y2) at (180:\size cm) [label=180:]{};

        \path
        (v1) edge[line width=\widthedge] (u2)
        (y1) edge[line width=\widthedge] (x2);     

        \path
        (u2) edge[green,line width=1,dotted] (y2)
        (v2) edge[green,line width=\widthedge] (u2)
        (x2) edge[green,line width=1,dotted] (v2)
        (y2) edge[blue,line width=\widthedge] (x2);

        \path (v2) edge[line width=.75] node[above=-2pt]{\small $e_1$} (y2);
        
        }{}
    }

\end{tikzpicture}
}

\renewcommand{\widthedge}{1.5}

\newcommand{\ProofBlowUp}[1]{
\begin{tikzpicture}[x=#1 cm, y=#1 cm]

    \vertex (v0) at (180:1) [label=180:$v_0$]{};
    
    \vertex (x) at (90:1) [label=90:$x$]{};
    \vertex (y) at (-90:1) [label=-90:$y$]{};
    
    \vertex (z) at (0:1) [label=0:$z$]{};

    \begin{scope}[on background layer]
        \draw[fill=gray!50, opacity=0.75] (0,0) ellipse (0.25 and 1.4) node[above=50pt,left=7pt]{$V_i$};
        \draw[fill=gray!50, opacity=0.75] (z) ellipse (0.25 and 1.4) node[above=50pt,right=7pt]{$V_j$};
    \end{scope}

    \path
    (v0) edge[red,line width=\widthedge] node[sloped,above]{colour $i$} (x)
    (v0) edge[red,line width=\widthedge] node[below]{$i$} (y);

    \path
    (z) edge[blue,dotted,line width=\widthedge] node[above=8pt]{$\phi(xz)$} (x)
    (z) edge[teal,dashed,line width=\widthedge] node[below=8pt]{$\phi(yz)$} (y);

\end{tikzpicture}}

\newcommand{\ProofMonoClique}[1]{
\begin{tikzpicture}[x=#1 cm, y=#1 cm]
    
    \vertex (x) at (135:1) [label=90:$x$]{};
    \vertex (y) at (-135:1) [label=-90:$y$]{};
    
    \vertex (u) at (45:1) [label=90:$u$]{};
    \vertex (v) at (-45:1) [label=-90:$v$]{};    

    \begin{scope}[on background layer]
        \draw[fill=gray!50,opacity=0.75] ({-cos(45)},0) ellipse (0.25 and 1.4) node[above=50pt,left=7pt]{$V_i$};
        \draw[fill=gray!50, opacity=0.75] ({cos(45)},0) ellipse (0.25 and 1.4) node[above=50pt,right=7pt]{$V_j$};
    \end{scope}

    \path
    (x) edge[red,line width=\widthedge]  (u)
    (y) edge[red,line width=\widthedge]  (v);

    \path
    (x) edge[blue,dotted,line width=\widthedge] node[left]{$\phi(xy)$} (y)
    (u) edge[teal,dashed,line width=\widthedge] node[right]{$\phi(uv)$} (v);

\end{tikzpicture}}

\tikzstyle{disc}=[circle, draw=gray, fill=gray!30, inner sep=0, minimum size=60pt,transform shape]

\newcommand{\HamcycleSwitchFree}[1]{

\begin{tikzpicture}[scale=#1]

      \pgfmathsetmacro\i{1}
      \pgfmathsetmacro\angle{90-(\i-1) * 360/6}
      \pgfmathsetmacro\rot{(1-\i) * 360/6}
      \begin{scope}[shift={(\angle:4)},rotate=\rot]
          \node[disc, label=above:$V_\i$] (V\i) at (0,0) {};
          \draw[thick,orange] plot[smooth] coordinates {(180:0.8) (140:0.6) (220:0.4) (70:0.8) (0:0.8)};
          \vertex (u\i) at (180:0.8) [label=-90:$u_\i$]{};
          \vertex (v\i) at (0:0.8) [label=-90:$v_\i$]{};
          \vertex (w\i) at (-90:0.8) [label=90:$w_\i$]{};
          \node at (110:0.6) {$Q_\i$};
      \end{scope}

      \pgfmathsetmacro\i{2}
      \pgfmathsetmacro\angle{90-(\i-1) * 360/6}
      \pgfmathsetmacro\rot{(1-\i) * 360/6}
      \begin{scope}[shift={(\angle:4)},rotate=\rot]
          \node[disc, label=above:$V_\i$] (V\i) at (0,0) {};
          \draw[thick,orange] plot[smooth] coordinates {(180:0.8) (140:0.6) (220:0.4) (70:0.8) (0:0.8)};
          \vertex (u\i) at (180:0.8) [label={[label distance=-3pt]-100:$u_\i$}]{};
          \vertex (v\i) at (0:0.8) [label={[label distance=-3pt]-45:$v_\i$}]{};
          \vertex (w\i) at (-90:0.8) [label={[label distance=-3pt]0:$w_\i$}]{};
          \node at (110:0.6) {$Q_\i$};
      \end{scope}

      \pgfmathsetmacro\i{3}
      \pgfmathsetmacro\angle{90-(\i-1) * 360/6}
      \pgfmathsetmacro\rot{(1-\i) * 360/6}
      \begin{scope}[shift={(\angle:4)},rotate=\rot]
          \node[disc, label=above:$V_\i$] (V\i) at (0,0) {};
          \draw[thick,orange] plot[smooth] coordinates {(180:0.8) (140:0.6) (220:0.4) (70:0.8) (0:0.8)};
          \vertex (u\i) at (180:0.8) [label={[label distance=-3pt]180:$u_\i$}]{};
          \vertex (v\i) at (0:0.8) [label={[label distance=-3pt]90:$v_\i$}]{};
          \vertex (w\i) at (-90:0.8) [label=135:$w_\i$]{};
          \node at (110:0.6) {$Q_\i$};
      \end{scope}

      \pgfmathsetmacro\i{4}
      \pgfmathsetmacro\angle{90-(\i-1) * 360/6}
      \pgfmathsetmacro\rot{(1-\i) * 360/6}
      \begin{scope}[shift={(\angle:4)},rotate=\rot]
          \coordinate (V\i) at (0,0);
          \vertex (u\i) at (180:0.8) [label=-90:$u_\i$]{};
          \vertex (w\i) at (-90:0.8) [label=90:$w_\i$]{};
      \end{scope}           

      \pgfmathsetmacro\i{4.5}
      \pgfmathsetmacro\angle{90-(\i-1) * 360/6}
      \pgfmathsetmacro\rot{(1-\i) * 360/6}
      \begin{scope}[shift={(\angle:4)},rotate=\rot]
          \node[transform shape] at (0,0) {\huge $\cdots\cdot\cdot$};
      \end{scope}  

      \pgfmathsetmacro\i{5}
      \pgfmathsetmacro\angle{90-(\i-1) * 360/6}
      \pgfmathsetmacro\rot{(1-\i) * 360/6}
      \begin{scope}[shift={(\angle:4)},rotate=\rot]
          \coordinate (Vprev) at (0,0);
          \vertex (vprev) at (0:0.8) [label=-90:$v_{s'-1}$]{};
          \vertex (wprev) at (-90:0.8) [label=-90:$w_{s'-1}$]{};
      \end{scope}        

      \pgfmathsetmacro\i{6}
      \pgfmathsetmacro\angle{90-(\i-1) * 360/6}
      \pgfmathsetmacro\rot{(1-\i) * 360/6}
      \begin{scope}[shift={(\angle:4)},rotate=\rot]
          \node[disc, label=above:$V_{s'}$] (Vs) at (0,0) {};
          \draw[thick,orange] plot[smooth] coordinates {(180:0.8) (140:0.6) (220:0.4) (70:0.8) (0:0.8)};
          \vertex (us) at (180:0.8) [label=0:$u_{s'}$]{};
          \vertex (vs) at (0:0.8) [label=-90:$v_{s'}$]{};
          \vertex (ws) at (-90:0.8) [label=90:$w_{s'}$]{};
          \node at (110:0.6) {$Q_{s'}$};
      \end{scope}

      \path 
        (v1) edge[thick,red] (u2)
        (v2) edge[thick,blue] (u3)
        (v3) edge[thick,teal] (u4)
        (vprev) edge[thick,blue] (us)
        (vs) edge[thick,teal] (w1)
        (w1) edge[thick,red] (w2)
        (w2) edge[thick,blue] (w3)
        (w3) edge[thick,teal] (w4)
        (wprev) edge[thick,blue] (ws)
        (ws) edge[thick,teal] (u1);

\end{tikzpicture}
}

\begin{document}

\author{
Simona Boyadzhiyska\,\footnotemark[1] \qquad 
Shagnik Das\,\footnotemark[2] \qquad
Thomas Lesgourgues\,\footnotemark[3] \\
Kalina Petrova\,\footnotemark[4]
}

\footnotetext[1]{HUN-REN Alfréd Rényi Institute of Mathematics, Budapest, Hungary.  Email: {\tt simona@renyi.hu}.}
\footnotetext[2]{Department of Mathematics, National Taiwan University, Taipei 10617, Taiwan. Email: {\tt shagnik@ntu.edu.tw}.}
\footnotetext[3]{Department of Combinatorics and Optimization, University of Waterloo, Canada. Email: {\tt tlesgourgues@uwaterloo.ca}.}
\footnotetext[4]{Institute of Science and Technology Austria (ISTA), Klosterneurburg 3400, Austria. Email: {\tt kalina.petrova@ist.ac.at}.}

\sloppy

\title{Odd-Ramsey numbers of Hamilton cycles}

\maketitle

\begin{abstract}
    The odd-Ramsey number $\ro{n}{H}$ of a graph $H$, as introduced by Alon in his work on graph-codes, is the minimum number of colours needed to edge-colour $K_n$ so that every copy of $H$ intersects some colour class in an odd number of edges. In this paper, we determine the odd-Ramsey number of Hamilton cycles up to a small multiplicative factor, proving that $\ro{n}{C_n} = \Theta(\sqrt{n})$. Our upper bound follows from an explicit finite-field construction, while the matching lower bound uses a combinatorial framework based on parity switches. We also initiate the study of odd-Ramsey numbers of Hamilton cycles in Dirac graphs, demonstrating that a small increase in the minimum degree beyond $n/2$ forces nontrivial odd-Ramsey numbers.
\end{abstract}

\section{Introduction}

Given graphs $G$ and $H$, the \emph{odd-Ramsey number} of  $H$ in $G$, denoted by $\ro{G}{H}$, is defined as the minimum number of colours $r$ in an edge-colouring $G = G_1 \cup G_2 \cup \hdots \cup G_r$ with the property that every copy of $H$ intersects some~$G_i$ in an odd number of edges. For simplicity, we say that every copy of $H$ has an \emph{odd colour class}, and when~$G$ is the complete graph~$K_n$, we often simplify our notation and write $\ro{n}{H}$ instead of~$\ro{K_n}{H}$. Note that, if~$G$ contains no copy of $H$, or if $H$ contains an odd number of edges, then we trivially have~$\ro{G}{H} = 1$.

Alon~\cite{alon2024graph} implicitly introduced the notion of odd-Ramsey numbers in his work on graph-codes. Given a graph $H$, an~\emph{$H$-code} is a family of graphs on a common set of vertices with the property that the symmetric difference of any two of its members is not isomorphic to $H$. Alon showed that, if $H$ has small odd-Ramsey number, then one can construct a relatively dense $H$-code.

As noted by Alon, the odd-Ramsey number of a graph is an interesting parameter to study in its own right, as it fits naturally into a sequence of Ramsey Theory variants. The classic Ramsey problem asks for the minimum number~$r(G,H)$ of colours needed to edge-colour $G$ so that every copy of $H$ receives at least two colours. Erd\H{o}s and Gy\'arf\'as~\cite{erdHos1997variant} generalised this by asking for the \emph{generalised Ramsey number} $f(G,H,q)$, the minimum number of colours needed to ensure that every copy of $H$ receives at least $q$ colours. This problem has attracted a great deal of attention over the past few decades, as researchers attempt to gain a deeper understanding of the Ramsey phenomenon; see, for instance,~\cite{axenovich2000generalizedbip,bennett2024generalized,bennett2024erdHos,conlon2015erdHos,eichhorn2000note,fox2009ramsey,mubayi1998edge}. The odd-Ramsey numbers, where we concern ourselves with the parity rather than the quantity of colours, are therefore closely related, both qualitatively and quantitatively. Indeed, one has the immediate upper bound $\ro{G}{H} \le f(G,H,\floor{\tfrac12 e(H)} +1)$, since if the number of colours in a copy of $H$ exceeds half the number of its edges, there must be a colour class of size one.

Despite only having been introduced recently, there is already a considerable number of results on odd-Ramsey numbers. In his original paper, Alon~\cite{alon2024graph} implicitly addressed the odd-Ramsey numbers of stars, matchings, and certain families of cliques. Several groups of authors have since considered the case of complete graphs, seeking to determine $r_{odd}(n,K_t)$. In 2023, Ge, Xu, and Zhang~\cite{ge2023new} conjectured that $r_{odd}(n,K_t) = n^{o(1)}$ for all $t$; before this conjecture, Cameron and Heath~\cite{cameron2023new} had proved that $r_{odd}(n,K_4) = n^{o(1)}$. Bennett, Heath, and Zerbib~\cite{bennett2023edgecoloring} and, independently, Ge, Xu, and Zhang~\cite{ge2023new} proved the $t=5$ case. Most recently, Yip~\cite{yip2024k8} showed in 2024 that the conjecture holds for $t=8$ (the cases $t=6,7$ are trivial as then $K_t$ has an odd number of edges). For an arbitrary graph $H$ with an even number of edges, Versteegen~\cite{versteegen2025upper} proved the general lower bound $r_{odd}(n,H) = \Omega( \log n)$, and showed that this could be strengthened to a polynomial lower bound if $H$ can be decomposed into independent sets in a particular way. Versteegen showed that almost all graphs with an even number of edges satisfy this property, and Janzer and Yip~\cite{janzer2024} proved sharp bounds on the probability of the random graph~$G(n,p)$ having such a decomposition. In a previous article~\cite{BDLP2024OddBip}, we studied the odd-Ramsey numbers of complete bipartite graphs. We obtained results for the family of all spanning complete bipartite graphs on $n$ vertices, its subfamilies, and for individual fixed graphs $K_{s,t}$. Finally, for a non-complete host graph $G$, Bennett, Heath, and Zerbib~\cite{bennett2023edgecoloring} proved in~2023 that $r_{odd}(K_{n,n},K_{2,2})=\frac12n+o(n)$. Very recently Crawford, Heath, Henderschedt, Schwieder, and Zerbib~\cite{crawford2025odd} extended this result to $r_{odd}(K_{n,n},K_{2,t})=\frac{1}{t}n+o(n)$ for all $t\geq 2$, and to some hypergraph variants.

In this paper, our target graph $H$ will be the \emph{Hamilton cycle} $C_n$. The study of Hamilton cycles is central to extremal graph theory, not least of all because Karp~\cite{karp2009reducibility} showed that the problem of determining whether or not a graph is Hamiltonian is NP-complete. Thus, as classifying Hamiltonian graphs is likely to be intractable, there is a significant body of research aiming to find sufficient conditions guaranteeing the existence of a Hamilton cycle in a graph. Perhaps the most well-known such result is due to Dirac~\cite{dirac1952some}, showing that every $n$-vertex graph $G$ with minimum degree at least $n/2$ contains a Hamilton cycle. As is by now standard in this area, we call an $n$-vertex graph with minimum degree at least $n/2$ a \emph{Dirac graph}. In the decades following Dirac's cornerstone work, researchers have sought to generalise and extend this result in various directions. For example, Nash-Williams~\cite{nash1971edge} proved that every Dirac graph on $n$ vertices contains linearly many edge-disjoint Hamilton cycles, which was sharpened to a tight bound of $(n-2)/8$ by Csaba, K\"{u}hn, Lo, Osthus, and Treglown~\cite{csaba2016proof}. We also know from the work of S{\'a}rk{\"o}zy, Selkow, and Szemer{\'e}di~\cite{sarkozy2003number} that every $n$-vertex Dirac graph contains at least $c^nn!$ many Hamilton cycles, for some constant~$c>0$.

When it comes to the odd-Ramsey number, note that if an $n$-vertex graph $G$ has no Hamilton cycle, or if $n$ is odd, then trivially $\ro{G}{C_n} = 1$. Thus, we shall restrict our discussion to the case when $n$ is even and the host graph~$G$ is Hamiltonian. As mentioned previously, the Erd\H{o}s--Gy\'arf\'as problem provides upper bounds on the odd-Ramsey problem, namely $\ro{n}{C_n} \le f(K_n, C_n, \tfrac12 n + 1)$. However, as we trivially have $f(G, H, q) \ge q$ (provided $H \subseteq G$), this can give a linear upper bound at best. In our main result, we show that $\ro{n}{C_n}$ is actually of much smaller order of magnitude. We then consider sparser host graphs --- in light of the aforementioned results about the plethora of Hamilton cycles in Dirac graphs, one can ask what happens in this sparser setting. More precisely, how large a minimum degree must a graph $G$ have to force $\ro{G}{C_n}$ to be large? We provide some initial results in this direction, and believe that this is a problem meriting further exploration.

\subsection{Results}

Our primary result determines the odd-Ramsey number of Hamilton cycles in complete graphs up to a constant factor.

\begin{theorem}\label{thm:complete_graphs}
For every even integer $n\geq 4$, we have 
\[\parens*{\frac{\sqrt{2}}{2}+o(1)} \sqrt{n} \leq r_{odd}(n,C_n) \leq \frac{3\sqrt{2}}{2} \sqrt{n},\]
where the $o(1)$ error term goes to $0$ as $n$ tends to infinity.
\end{theorem}

We next investigate the behaviour of odd-Ramsey numbers of Hamilton cycles in Dirac graphs.  Specifically, we are interested in how small the odd-Ramsey number can be among all host graphs of given minimum degree. To this end, we define
\begin{align*}
    \rod{n}{\delta}{C_n} = \min \{ \ro{G}{C_n}: v(G) = n, \delta(G) \ge \delta \}.
\end{align*}
Note that this new function is a generalisation of $\ro{n}{C_n}$, since $\ro{n}{C_n} = \rod{n}{n-1}{C_n}$. When $n$ is even, any Hamiltonian $n$-vertex graph $G$ must have $\ro{G}{C_n} \ge 2$. However, since Dirac graphs have many Hamilton cycles, it is natural to ask whether they must have much larger odd-Ramsey numbers. By \cref{thm:complete_graphs}, it follows that $\rod{n}{\delta}{C_n}\leq  \ro{n}{C_n}\leq \frac{3\sqrt{2}}{2}\sqrt{n}$ for any $\frac n2 \leq \delta \leq n-1$. However, our next statement shows that there exist more efficient colourings whenever~$\delta < n - \sqrt{2n}$; in particular, $\rod{n}{\tfrac{n}{2}}{C_n} = 2$.

\begin{proposition}\label{prop:UBSparse}
    For every even integer $n\geq 4$ and every $\delta\in \big[\frac{n}{2},n\big]$, we have
    \[r_{odd}(n,\delta;C_n)\leq \min\Big\{
    2\delta-n+2,\
    \frac{3\sqrt{2}}{2}\cdot \frac{\delta}{\sqrt{n}}+3
    \Big\}.\]
\end{proposition}

In our final result, we show that while minimum degree $n/2$ is not enough to force a nontrivial odd-Ramsey number, a very small increase in the minimum degree does the trick.
\begin{theorem}\label{thm:Dirac+const}
    For every even integer $n\geq 4$, we have $\rod{n}{\frac n2+4}{C_n}\geq 3$.
\end{theorem}

While these results shed some light on the function $\rod{n}{\delta}{C_n}$, its true behaviour is far from being known, and in~\cref{sec:ConcludingRemarks} we state  several open questions that we would like to see solved.

\paragraph{Notation}
Our notation is mostly standard. Unless otherwise specified, the term colouring refers to an edge-colouring and an $r$-colouring uses the colour-palette $[r]$. For brevity, we sometimes refer to the edges of some colour~$c$ as \emph{$c$-edges}. For the sake of simplicity, we say that a subgraph of an edge-coloured graph $G$ is \emph{odd-coloured} if it has an odd colour class, and \emph{even-coloured} otherwise. 
The \emph{length} of a path is the number of edges it contains. Given vertices $v_1,v_2,\ldots,v_k$, we denote by $v_1v_2\ldots v_k$ the path with edge-set $\set*{v_i v_{i+1} \colon i\in[k-1]}$, and by $v_1v_2\ldots v_kv_1$ the cycle with edge-set $\set*{ v_i v_{i+1} \colon i \in [k-1]}\cup\set{ v_k v_1 }$. If $P$ is a path between the vertices $v_j$ and $v_{j + \ell}$, we write $v_1v_2\ldots v_{j-1}v_jPv_{j+\ell}v_{j+\ell+1}\ldots v_k$ for the path $v_1v_2\ldots v_k$. 
Given vertices $a,b$, an \emph{$\set{a,b}$-path} is a path between its endpoints $a$ and $b$. A {\em cherry} is a  path of length two. Given two not necessarily disjoint subsets of vertices~$U,W\subseteq V(G)$, we write $(U,W)$ for the set of edges in $G$ with one endpoint in $U$ and one endpoint in $V$.

\paragraph{Organisation of the paper.} We prove the upper bound of~\cref{thm:complete_graphs} in~\cref{sec:UpperBound}, and its lower bound in~\cref{sec:LowerBound}.~\cref{sec:SparserGraphs} is devoted to the study of sparser host graphs and the proofs of~\cref{prop:UBSparse,thm:Dirac+const}. In \cref{sec:ConcludingRemarks}, we propose a number of open problems and directions for further study in this area.

\section{Upper bound for complete graphs}\label{sec:UpperBound}

The upper bound from~\cref{thm:complete_graphs} is a direct consequence of the following proposition.

\begin{proposition}\label{prop:UpperboundComplete}
    For any integers $m\geq 1$ and $t\geq0$, we have $r_{odd}(m2^t, C_{m2^t}) < 2^t + m$. It follows that, if $n=4^t$ for some integer $t\geq1$, we have $r_{odd}(n, C_{n}) < 2\sqrt{n}$. In general, for every even integer $n$, we have $r_{odd}(n, C_n) \leq \frac{3\sqrt{2}}{2} \sqrt{n}$. 
\end{proposition}

\begin{proof}
    Let $m, t$ be positive integers and assume first that we have~$n = m2^t$. We take the vertex set of the complete graph~$K_n$ to be ${\mathbb{F}_2^t} \times [m]$, and we colour its edges with the~$2^t + m-1$ colours in~$(\mathbb{F}_2^t \cup [m])\setminus\{1\}$ as follows. For every~$\vec{v}\in\mathbb{F}_2^t$ and every~$x\in[m]$, we start by colouring the edge between~$(\vec{0}, m)$ and~$(\vec{v}, x)$ with the colour~$x$.  We colour all other edges between~$(\vec{v}_1, x_1)$ and $(\vec{v}_2, x_2)$ with~$\vec{v}_1 + \vec{v}_2 \in \mathbb{F}_2^t$. We then remove the colour $1$ by identifying it with the colour $\vec{0}$, that is, for every $\vec{v}\in\mathbb{F}_2^t$, the edge between~$(\vec{0}, m)$ and~$(\vec{v}, 1)$ is given colour~$\vec{0}$ instead of $1$.
    
    Suppose for a contradiction that there exists an even-coloured Hamilton cycle $C$ in this colouring. Assume first that the two edges incident to $(\vec{0},m)$ in $C$ are coloured with distinct colours. At least one such edge is coloured with~$x\in\{2,\ldots,m\}$, and this colour does not appear on any other edge in $C$, a contradiction. It follows that the two edges of $C$ incident to $(\vec{0}, m)$ must be of the same colour. That is, there is some $x \in [m]$, and distinct~$\vec{v}_1, \vec{v}_2 \in \mathbb{F}_2^t$ such that $(\vec{0}, m)$ is neighboured by $(\vec{v}_1, x)$ and $(\vec{v}_2, x)$ on $C$. Let $P$ be the path between~$(\vec{v}_1, x)$ and~$(\vec{v}_2, x)$ obtained by deleting $(\vec{0}, m)$ from $C$.

    If $C$ was even-coloured, then $P$ must be as well. Thus, the sum of the colours appearing on the edges of $P$ (which are elements of $\mathbb{F}_2^t$) must be $\vec{0}$. However, this sum is $\sum_{ \{(\vec{u}_1, x_1), (\vec{u}_2, x_2) \} \in E(P)} (\vec{u}_1 + \vec{u}_2)$, and every internal vertex~$(\vec{u}, x)$ contributes $2 \vec{u} = \vec{0}$. Thus, the sum is equal to $\vec{v}_1 + \vec{v}_2$, but we know that $\vec{v}_1\neq \vec{v}_2$, giving the desired contradiction. We obtain $r_{odd}(m2^t, C_{m2^t}) \leq 2^t + m-1$, as desired.\smallskip
    
    Now, if $n=4^t$, we take $m=2^t$, so that $n=m2^t$. The result then follows, as $r_{odd}(m2^t, C_{m2^t}) < 2^t + m = 2\cdot2^{t}=2\sqrt{n}$.
    Finally, to derive the bound for general $n$, let $t \in \set{ \floor*{\frac12 \log_2 n}, \ceil*{\frac12 \log_2 n}}$ be the integer closest to $\frac12 \log_2 n$, and let~$m = \ceil*{n/2^t}$.  Observe that $n \le m2^t$ and $\frac{\sqrt{n}}{2^t} \in \left[ \frac{\sqrt{2}}{2}, \sqrt{2} \right]$. As $x + \frac{1}{x} \le \frac{3\sqrt{2}}{2}$ for $x$ in this interval, we have 
    \[2^t+m -1\leq \Big(\frac{2^t}{\sqrt{n}}+\frac{\sqrt{n}}{2^t}\Big) \sqrt{n} \leq \frac{3\sqrt{2}}{2} \sqrt{n}.\]
    We can thus embed the $n$ vertices of $K_n$ into the $m2^t$ vertices of the above construction, ensuring that $(\vec{0}, m)$ is in the image. The above proof then shows that there is no even-coloured Hamilton cycle.    
\end{proof}

We observe that the constant  $\frac{3\sqrt{2}}{2}$ is best possible with this construction, as it is attained when~$n=2^{2t+1}$ for some integer $t$.

\section{Lower bound for complete graphs}\label{sec:LowerBound}

In this section we prove the lower bound of \cref{thm:complete_graphs}. We begin by introducing our main tool, \emph{switches}, and give an outline of the proof in Section~\ref{subsec:outline}. In Section~\ref{subsec:auxiliary}, we state and prove some helpful lemmata, and in Section~\ref{subsec:proof_lower_bound}, we provide the proof of our lower bound.

\subsection{Main tool and proof outline}
\label{subsec:outline}

Our proof makes extensive use of switches, which allow us to flip the parities of colours, one pair at a time. Similar gadgets were used, for instance, in~\cite{CNP24}.
\begin{definition}[Switch]
    Given an $r$-colouring of $K_n$, let $c_1\neq c_2$ be two colours. A  {\em $\set{c_1,c_2}$-switch} is a four-cycle~$S = uvxyu$ such that the number of $c$-edges  in $S$ is odd if $c \in \{c_1, c_2\}$, and even otherwise. 
    A four-cycle is a \emph{switch} if it is a $\set{c_1,c_2}$-switch for some distinct colours $c_1,c_2$.
\end{definition}

A \{red,blue\}-switch is illustrated in \cref{fig:switch}. 
Let $S =uvxyu$ be a switch and~$P$ be a path between~$v$ and~$y$ containing neither~$u$ nor~$x$. Then, as illustrated in~\cref{fig:switchoption1,fig:switchoption2}, we can find not one but two $\set{u,x}$-paths containing the path~$P$ and the vertices of~$S$: $uvPyx$ and $uyPvx$. Crucially, the parity of the number of $c_1$-edges and $c_2$-edges changes between the two paths, while the parities of all other colours remains the same.  
We introduce some terminology to allow us to refer to these different ways of traversing the edges of the switch.

\begin{definition}[Base/flipped matchings]
    Given an $r$-colouring of $K_n$ and a switch $S=uvxyu$, we call the matching~$\set{uv, xy}$ the \emph{base matching} of $S$, and the matching $\set{uy,vx}$ the \emph{flipped matching} of $S$.
\end{definition}

\def\sizeGraph{1.5}
\begin{figure}[ht]
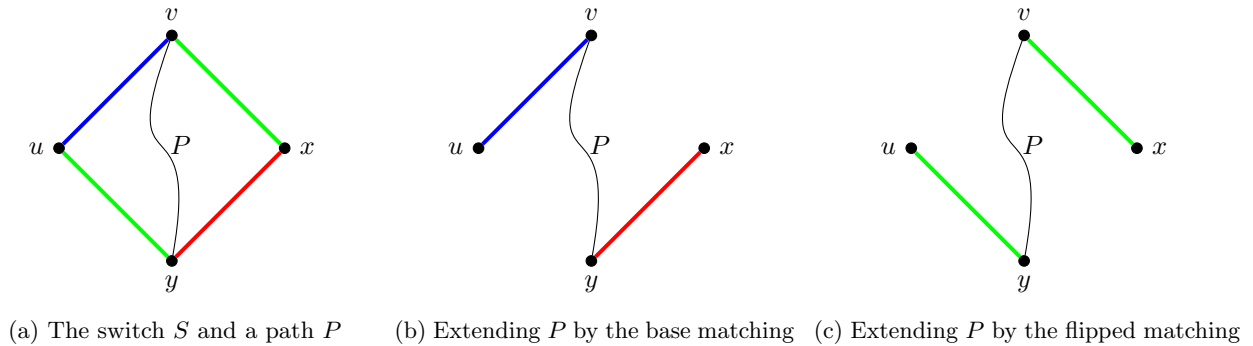

    \centering

    \begin{subfigure}{0.3\textwidth}
        \centering
        \switch{\sizeGraph}{0}
        \caption{The switch $S$ and a path $P$}
        \label{fig:switch}
    \end{subfigure}%
    \hspace{0.3cm}%
    \begin{subfigure}{0.3\textwidth}
        \centering
        \switch{\sizeGraph}{1}
        \caption{Extending $P$ by the base matching}
        \label{fig:switchoption1}
    \end{subfigure}%
    \hspace{0.3cm}%
    \begin{subfigure}{0.32\textwidth}
        \centering
        \switch{\sizeGraph}{2}
        \caption{Extending $P$ by the flipped matching}
        \label{fig:switchoption2}
    \end{subfigure}     
    \caption{Using a switch $S$}
    \label{fig:UseSwitch}
\end{figure}

In order to prove a lower bound on $\ro{n}{C_n}$, we consider an arbitrary $r$-colouring of $K_n$ for some $r$ and show that it must contain an even-coloured Hamilton cycle. This is where switches come in very handy. On the one hand, if we can find switches for many pairs of colours and manage to embed them into a Hamilton cycle, we will be able to use them to ensure the colours appear with the correct parity. On the other hand, we will see that a lack of switches gives us enough structural information about the colouring to directly construct an even-coloured Hamilton cycle.
In practice, making this work requires some additional ideas and more technical details. In what follows, we give a more in-depth proof overview.

We start by iteratively finding vertex-disjoint switches and setting them aside, until none are left. Recall that a~$\{c_1,c_2\}$-switch enables us to change the parities of the colours $c_1$ and $c_2$ simultaneously. Whenever we find and set aside a~$\{c_1,c_2\}$-switch $S$, we \emph{merge} the colours $c_1$ and $c_2$ and treat them as the same colour until the final step of the proof. The problem is now reduced to finding an even-coloured Hamilton cycle (with respect to this new colouring) that passes through the base matching of $S$. At the end, we can \emph{unmerge} the colours $c_1$ and $c_2$. Since they each appear with the same parity on our Hamilton cycle, if they both appear an odd number of times, we switch to traversing the flipped matching of $S$ instead, fixing their parities.

If we find sufficiently many vertex-disjoint switches to iteratively merge all colours into one single colour, we are done, as we can then pick any Hamilton cycle that passes through all the base matchings of these switches and iteratively unmerge all colours (in reverse order) as explained above. It remains to discuss what to do if we cannot find enough switches to do that. Indeed, our initial colouring may even have no switches at all.

To handle this case, we prove an auxiliary result stating that any switch-free $r$-colouring of a complete graph exhibits a very particular structure. Namely, all but one of its vertices can be partitioned into at most $r$ monochromatic cliques of the same colour such that the bipartite graphs between these cliques are monochromatic. It is not too hard to see that such a structure has an even-coloured spanning cycle. To extend this to an even-coloured Hamilton cycle in $K_n$, we need to incorporate the switches we previously set aside. We will then be able to use these switches to unmerge any colours we merged before arriving at our switch-free colouring.

Unfortunately, it may not always be possible to embed the switches into the cycle without introducing an odd colour. For instance, it may be that one of the endpoints of a switch is only incident to edges of a single colour, and that colour does not appear anywhere else in the graph. Absorbing this switch into our cycle would then require us to use that colour precisely once, which we are not allowed to do. Thus, we cannot simply find any set of switches and hope to be able to embed them into the almost-spanning cycle.

Instead, we choose the switches we set aside carefully and incorporate them into a special structure that we refer to as a \emph{SPICy} (short for \emph{Switches Placed in Cycles}). A SPICy consists of a set of vertex-disjoint cycles with some switches embedded in them. We maintain the invariant that the SPICy (including the base matchings of the embedded switches) always has an even-coloured $2$-factor, which we show is possible via a linear-algebraic argument.
The SPICy makes it easy to incorporate the switches into our cycle --- using the pigeonhole principle, we find some \emph{bridges}, which are monochromatic pairs of edges connecting each cycle of the SPICy to the almost-spanning cycle in an appropriate way.

Throughout this process, we will have to perform some ``cleaning up'' steps, which will involve temporarily getting rid of some vertices. At the end, we will absorb those into the almost-spanning cycle using monochromatic cherries at carefully chosen places.

In summary, with the help of the SPICy, we are able to absorb the switches into our almost-spanning cycle, obtaining an even-coloured Hamilton cycle. As previously mentioned, we finish the proof by unmerging all merged colours and fixing their parities as needed via the switches.

\subsection{Auxiliary results}
\label{subsec:auxiliary}

A key step in our proof of the lower bound of~\cref{thm:complete_graphs} is the following statement, proving that any coloured complete graph containing no switches exhibits a specific structure.

\begin{lemma}\label{lem:noswitches}
    Let $n,r\geq 1$ be integers, $G = K_n$, $v_0\in V(G)$ be an arbitrary vertex, and $\phi: E(G) \to [r]$ be an $r$-edge-colouring of $G$ that does not contain a $\set{c_1,c_2}$-switch for any colours~$c_1 \neq c_2$. Then the vertex set of~$G-v_0$ can be partitioned into (possibly empty) sets~$V_1,\dots,V_r$ such that $G[V_1], \dots, G[V_r]$ induce monochromatic cliques of the same colour and the bipartite graph $(V_i,V_j)$ between any pair is monochromatic.
\end{lemma}

\begin{proof}
    For every $i\in[r]$, let~$V_i$ be the set of vertices~$v$ for which~$\phi(v_0v)=i$.

    Suppose that, for some (not necessarily distinct) $i,j\in [r]$, there are two edges of different colours between~$V_i$ and~$V_j$. Then there exists a vertex $z\in V_i\cup V_j$ and edges $zx,zy\in (V_i,V_j)$ such that $\phi(zx)\neq \phi(zy)$ (see~\cref{fig:ContraBlowup}). But~$\phi(v_0x) = \phi(v_0y)$, so $v_0xzyv_0$ is a $\set{\phi(zx),\phi(zy)}$-switch, a contradiction.

    In particular, taking $i = j$ shows that the cliques $G[V_i]$ must be monochromatic. We now claim that they all have the same colour. Suppose $V_i$ and $V_j$ have different colours; in particular this means that $\size{V_i},\size{V_j}\geq 2$. Let $x,y \in V_i$ and $u, v \in V_j$ be distinct vertices, see \cref{fig:ContraMono}. We know that $\phi(xu) = \phi(yv)$, so if $\phi(xy)\neq \phi(uv)$, then $xuvyx$ is a~$\set{\phi(xy),\phi(uv)}$-switch, a contradiction to our assumption. 
\end{proof}

\def\sizeGraph{1.6}
\begin{figure}[ht]
    \centering

    \begin{subfigure}{0.40\textwidth}
        \centering
        \ProofBlowUp{\sizeGraph}
        \caption{Proof that each $(V_i,V_j)$ is monochromatic}
        \label{fig:ContraBlowup}
    \end{subfigure}%
    \hspace{0.3cm}%
    \begin{subfigure}{0.40\textwidth}
        \centering
        \ProofMonoClique{\sizeGraph}
        \caption{Proof that all cliques have the same colour}
        \label{fig:ContraMono}
    \end{subfigure}         

    \caption{Proof of~\cref{lem:noswitches}}
    \label{fig:ContradictionBlowupProof}

\end{figure}
     
\begin{remark}\label{rem:structural}
   We remark that, with a more involved proof, one can prove a stronger statement: up to a permutation of the colours, $G$ is the blow-up of a proper $(r-1)$-edge-colouring of $K_r$, where each vertex is blown up into a (possibly empty) monochromatic clique coloured $r$. However,~\cref{lem:noswitches} is sufficient for our purposes, and simpler to prove. We crucially use that the host graph is complete; for sparser host graphs, even this simpler characterisation of switch-free colourings does not hold, introducing additional difficulties in that setting.
\end{remark}

As mentioned in the proof outline, in the first step of our lower bound proof, we will build the ``switching structure'' that we refer to as a SPICy. Our SPICy will consist of a number of cycles with some switches embedded in them. The idea is that we will build a Hamilton cycle ``around'' this structure and then use the switches embedded in the SPICy to adjust the parities of some pairs of colours. Here we define and show the existence of a suitable base for this structure; we refer to it as a \emph{SPICy starter}. The SPICy starter is a set of vertex-disjoint cycles, to which we will iteratively attach switches to obtain our final SPICy.

\begin{definition}[SPICy starter]\label{def:spicystarter}
    For integers $r\geq 2$ and $t\geq 1$, an \emph{$(r,t)$-SPICy starter} in an $r$-coloured graph~$G$ is a collection of $t$ vertex-disjoint  cycles $C_1,\ldots,C_t$ such that each $C_i$ is an even-coloured cycle containing a monochromatic path of length at least $2r$. 
    The {\em size} of the SPICy starter is the sum of the cycles' lengths.
\end{definition}

Observe that an $(r,t)$-SPICy starter has size exceeding $2rt$.  In our proof, we will also need that the SPICy starter is not too large. The following lemma shows that this is indeed possible, provided~$n$ is large enough. 

\begin{lemma}\label{lem:spicystarter}
For all integers $r\geq 2$ and $n\geq 2r^2+6r^{3/2}$, let $G$ be an $r$-coloured complete graph on $n$ vertices. Then~$G$ contains an $(r,t)$-SPICy starter of size at least $2rt+2r$  and at most~$2rt+6r$ for some integer $t \le \ceil*{r^{1/2}}$. 
\end{lemma}

We shall in fact build our SPICy starter by finding a sequence of monochromatic cycles of even lengths. While the well-known theorem of Erd\H{o}s and Gallai~\cite{erdosgallai59} asserts that sufficiently dense graphs contain long cycles, the following extension due to Li and Ning~\cite{li2023eigenvalues}, based on a theorem of Vo{\ss} and Zuluaga~\cite{voss1977maximale}, ensures that they can be taken to be of even length.

\begin{lemma}[Li--Ning~\cite{li2023eigenvalues}] \label{lem:longevencycles}
    Let $G$ be a graph of average degree $d \ge 3$. Then $G$ contains an even cycle of length at least $d$.
\end{lemma}

Armed with this lemma, we can construct our SPICy starters.

\begin{proof}[Proof of~\cref{lem:spicystarter}]

    We start by building monochromatic vertex-disjoint even cycles $C_1,C_2,\ldots$ iteratively. For technical reasons, we set aside an arbitrary vertex $x\in V(G)$. We then let~$U_1:=V(G)\setminus\set{x}$. For~$i\geq 1$, assume that we have already chosen vertex-disjoint cycles $C_1,\ldots,C_{i-1}$ and set
    \[U_i:=U_1\setminus\bigcup_{1\leq j<i}C_j.\]
    We further assume that $C_j$ was the longest available cycle in $G[U_j]$.
    To find our next cycle, let $H_i \subseteq G[U_i]$ be a densest colour class. The average degree of $H_i$ is then at least $\frac{1}{r}(\card{U_i} - 1)$. Hence, provided $\card{U_i} \ge 2r^2 + 2r^{3/2} + 1$,~\cref{lem:longevencycles} guarantees the existence of an even cycle of length at least $2r + 2r^{1/2}$ in $H_i$. Let $C_i$ be a longest monochromatic even-length cycle in $G[U_i]$, and notice that, by our choice of $C_{i-1}$, we have $v(C_i)\leq v(C_{i-1})$.
    If $\sum_{j\leq i}|C_j|\geq 2ri+2r$, then we stop and set $t:=i$; otherwise, we continue to the next iteration.
    
    Observe that each cycle has length at least $2r+2r^{1/2}$, and therefore we have $t\leq \ceil*{r^{1/2}}$. Furthermore, for every~$i\leq t$, we have $\sum_{j< i}|C_j|\leq 2r(t-1)+2r-1\leq 2r^{3/2}+2r$, and therefore $|U_i| \geq n-1 - 2r^{3/2}-2r \geq 2r^2 +2r^{3/2}+1$, ensuring that the process does not stop prematurely.

    Now, if \(\sum_{j\leq t}|C_j|\leq 2rt+6r\), we simply take $C_1,\dots, C_t$ as our SPICy starter, since our stopping condition ensures that \(\sum_{j\leq t}|C_j|\geq 2rt+2r\). So assume otherwise.
    We claim that this can only occur if $t=1$. Indeed, we know that 
    \[ 2r(t-1)+2r^{1/2}(t-1) \leq \sum_{j<t}|C_j|<2r(t-1)+2r,\]
    so the final cycle $C_t$ has length more than $2rt+6r-2r(t-1)-2r \geq 4r$. But $v(C_1)\geq v(C_t)\geq 4r = 2r+2r$, so the process terminates with $t=1$.
    Hence, we assume $C_1$ has length more than $8r$. We argue that $C_1$ can be suitably shortened to have length between $4r$ and $8r$. Suppose $C_1=u_0\dots u_{\ell}u_0$ for some $\ell\geq 8r$. Let $x$ be the vertex we set aside at the start. Consider the edges $xu_{4i}$ for $i\in \set{0,\dots, r}$. Two of these edges, say $xu_{4i}$ and $xu_{4i'}$ for some~$0\leq i<i'\leq r$, have the same colour, so we can shorten the cycle by replacing the path $u_{4i}u_{4i+1}\dots u_{4i'}$ by the monochromatic cherry $u_{4i}xu_{4i'}$. This shortens the cycle by at least two and at most $4r-2$ vertices, so its length is still at least $8r-(4r-2) =4r+2$, and it still contains the monochromatic path~$u_{4r}\ldots u_{\ell}$. We repeat this process as long as the cycle has length exceeding~$8r$, while ensuring that the cycle obtained at each step still contains the required monochromatic path of length $2r$. In preparation for each subsequent iteration, we pick one of the vertices just discarded from the cycle arbitrarily and let that be the new vertex $x$. As we always replace an even-coloured path in $C_1$ by a monochromatic cherry, we end up with an even-coloured cycle of length at most $8r$ containing a monochromatic path of length at least $2r$, as desired.
\end{proof}

\subsection{Proof of the lower bound}
\label{subsec:proof_lower_bound}

In this section, we prove the lower bound in 
\cref{thm:complete_graphs}. We prove the following slightly stronger result, which immediately implies the required bound when $n$ is even.

\begin{theorem}
    \label{thm:SqrtLowerBound}
    If $r\geq 2$ and $n \ge 2r^2+40r^{3/2}$, then every $r$-colouring of $K_n$ contains a Hamilton cycle with at most one odd colour class. In particular, for even $n$ tending to infinity, we have 
    \[ \ro{n}{C_n} \geq \parens*{\frac{\sqrt{2}}{2} + o(1)}\sqrt{n}. \]
\end{theorem}

\begin{proof}[Proof of~\cref{thm:SqrtLowerBound}] Let $n$ be as given, and fix an r-colouring $\varphi$ of $K_n$. Our proof strategy is as follows. We will first build a structure potentially containing some switches (our SPICy), then show that the remainder of the graph has a very particular structure (using~\cref{lem:noswitches}). We shall then build an even-coloured cycle covering most of the vertices in the latter part of the graph, which can ``absorb'' the remainder of the vertices, including the SPICy. Finally, we will adjust any incorrect colour  parities by ``flipping'' some switches. 

\paragraph{Building a SPICy.}
By \cref{lem:spicystarter}, as $n\geq 2r^2+6r^{3/2}$, there exists some $t \le \ceil*{r^{1/2}}$ and an $(r,t)$-SPICy starter $C_1,\dots, C_t$ of size between $2rt+2r$ and $2rt+6r$. For each $\ell \in [t]$, let~$P_{\ell}$ denote a monochromatic path of length~$2r$ on $C_{\ell}$. To initialize the SPICy-building process, we let $C^{(0)}$ be our SPICy starter and choose $e_1, e_2, \hdots, e_r$ to be distinct edges on the SPICy starter that are not on the monochromatic paths $P_{\ell}$; these edges exist since $\sum_{\ell \leq t}\size{C_{\ell} \setminus P_{\ell}} \geq 2rt+2r - 2rt = 2r$. These will be our initial \emph{attachment edges}. Throughout the SPICy-building process, we maintain a $2$-factor $F^{(i)}$ of the SPICy, along with the invariant that $F^{(i)}$ is even-coloured. We initialize~$F^{(0)}:= C^{(0)}$. This 2-factor provides the initial ``route'' for traversing the components of the SPICy when we join them into a Hamilton cycle later; we might later alter some parts when we flip some switches. We also maintain a register $\cL$ of which pairs of colours get merged, so we can later unmerge them in the correct order, and keep track of the number $s^{(i)}$ of colours remaining, setting $s^{(0)} := r$. 

Now assume, for some $i \ge 0$, we have completed $i$ rounds of attaching switches, resulting in a SPICy $C^{(i)}$, together with a spanning even-coloured $2$-factor $F^{(i)}$ passing through the base matchings of all the switches in $C^{(i)}$ and all the attachment edges $e_1, e_2, \hdots, e_r$. Let~$s^{(i)}$ be the number of colours in the current colouring.

\noindent\underline{Case 1}: Suppose we find $s^{(i)}$ vertex-disjoint switches $S_1, \hdots, S_{s^{(i)}}$ in $G - V(C^{(i)})$. For each~$1 \le j \le s^{(i)}$, write $S_j = u_jv_jx_jy_ju_j$ and~$e_j = a_jb_j$, and let $\vec{w}_j \in \mathbb{F}_2^{s^{(i)}}$ be the vector recording the number of times each colour appears on the following six edges: $e_j$, $u_ja_j$, $x_jb_j$, $u_jv_j$, $v_jy_j$, and $x_jy_j$. Note that each vector~$\vec{w}_j$ records the colours of six edges, and hence these vectors belong to the hyperplane~$\vec{1}\cdot\vec{x} = 0$  in~$\mathbb{F}_2^{s^{(i)}}$. Since we have $s^{(i)}$ vectors, there is a nontrivial relation $\sum_{j\in J}\vec{w}_j = \vec{0}$ for some $J\subseteq [s^{(i)}]$.

Now that we have decided which switches to embed, we must update the SPICy and the corresponding $2$-factor. For each $j \in J$, consider the attachment edge $e_j = a_jb_j$, and let $C$ be the cycle in $F^{(i)}$ to which the edge belongs. We remove the edge $e_j$, and attach the switch $S_j$ in its place by adding the edges $a_j u_j$ and $x_j b_j$. We also add the edge~$v_j y_j$ to the SPICy, setting it as the new attachment edge $e_j$. For the spanning $2$-factor, after adding these vertices, we replace the edge $a_jb_j$ in the cycle with the path through the base matching of $S_j$; that is, $a_j u_j v_j y_j x_j b_j$. After performing these updates for each $j \in J$, we have our new SPICy $C^{(i+1)}$, together with its spanning $2$-factor~$F^{(i+1)}$.

Note that the vector $\vec{w}_j$ records how the colour parities in the $2$-factor change when we replace the edge $e_j$ with the path through the base matching of $S_j$. Since $\sum_{j\in J} \vec{w}_j = \vec{0}$, and $F^{(i)}$ was even-coloured, it follows that $F^{(i+1)}$ is also even-coloured. 

Once our selected switches have been embedded, we merge the corresponding pairs of colours, as we shall later be able to use these switches to correct their parities. More precisely, for $j \in J$, if $S_j$ is a $\set{c,c'}$-switch with $c<c'$, we merge the  colours $c$ and $c'$; that is, we change the colour of all~$c'$-edges to $c$, and add the pair $(c,c')$ to our register $\cL$. Note that at least one pair of colours is merged, and so $s^{(i+1)} < s^{(i)}$. We can now proceed to the next iteration of our SPICy-building process.

The evolution of the SPICy, represented for simplicity as a single cycle, is illustrated in \cref{fig:SwitchesInception}. Note that the $2$-factor $F^{(i)}$ is depicted with solid lines, while the flipped matchings of the switches are shown as dotted lines.

\def\sizeGraph{.8} 
\def\SubSize{0.23\textwidth}
\def\Space{0.1cm}

\begin{figure}[ht]
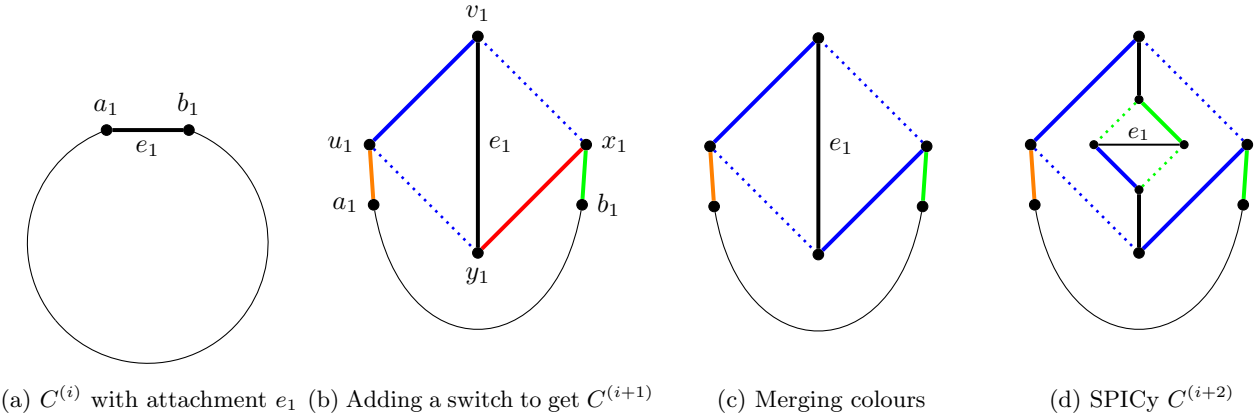

    \centering

    \begin{subfigure}{\SubSize}
        \centering
        \inception{\sizeGraph}{0}
        \caption{$C^{(i)}$ with attachment $e_1$}
        \label{fig:InitSpicy}
    \end{subfigure}%
    \hspace{\Space}%
    \begin{subfigure}{0.26\textwidth}
        \centering
        \inception{\sizeGraph}{1}
        \caption{Adding a switch to get~$C^{(i+1)}$}
        \label{fig:OneSwitch}
    \end{subfigure}
    \hspace{\Space}%
    \begin{subfigure}{\SubSize}
        \centering
        \inception{\sizeGraph}{2}
        \caption{Merging colours}
        \label{fig:Merging}
    \end{subfigure}  
    \hspace{\Space}%
    \begin{subfigure}{\SubSize}
        \centering
        \inception{\sizeGraph}{3}
        \caption{SPICy $C^{(i+2)}$}
        \label{fig:InceptedSwitch}
    \end{subfigure}      
    \caption{Evolution of a SPICy}
    \label{fig:SwitchesInception}

\end{figure}

If we can complete this process until there is only one colour remaining (at most $r-1$ iterations needed), then we have found a full set of switches. We then construct a Hamilton cycle incorporating these switches by taking an arbitrary path through all the vertices outside the SPICy, cutting each cycle of the 2-factor at an arbitrary point of the monochromatic path $P_i$, and joining these together, again in an arbitrary fashion. We can then skip to the final step, Unmerging Colours.

\noindent\underline{Case 2}: If we do not have $s^{(i)}$ vertex-disjoint switches, let $S_1, S_2, \hdots, S_q$ be a maximal set of vertex-disjoint switches, for some $q < s^{(i)}$. We then let $R = \bigcup_{j\in[q]}S_j$ be the set of vertices of these switches; we will refer to $R$ as the set of GROVes (short for \emph{Got Rid Of Vertices}). It follows that $\size{R} \leq 4q < 4r$. We then terminate the process and proceed to the next phase, setting $C^* = C^{(i)}$, $F^* = F^{(i)}$, and $s = s^{(i)}$.  Note that we added at most~$s^{(j)}$ switches to the SPICy $C^{(j)}$ to form $C^{(j+1)}$, and so the number of vertices in the final SPICy $C^*$ is at most 
\[ 2rt+6r + \sum_{j=0}^{i-1} 4s^{(j)} \leq 2r\ceil*{r^{1/2}} + 6r + \sum_{\ell = s+1}^r 4 \ell \leq 2r^{3/2} + 8r + 4\binom{r}{2} - 4\binom{s}{2} \leq 2r^2 + 10r^{3/2} - 2s^2.\]

\paragraph{Finding structure in the leftovers.}
The $s$-coloured graph $G-(V(C^{*})\cup R)$ contains no switches.  By~\cref{lem:noswitches}, the vertex set of this graph, save for an arbitrary vertex $v_0$, can be partitioned into sets $V_1,\dots, V_s$, such that each~$G[V_i]$ is monochromatic of the same colour, and each bipartite graph $(V_i,V_j)$ is monochromatic. 
Without loss of generality, assume the colour $1$ is used on all edges within the parts. Additionally, add $v_0$ to the set $R$ of GROVes, together with any part $V_i$ consisting of a single vertex. At this point, we have $\size{R} \leq  5r$. Without loss of generality, for some $s' \le s$, assume $V_1,\dots, V_{s'}$ are the parts containing at least two vertices. Notice that $\bigcup_{i\in [s']} V_i$ has size~$n-\size{V(C^*)} - \size{R} \geq 2s^2 + 25r^{3/2}$.

\paragraph{Attaching the SPICy.} The next step is to identify how we shall attach the SPICy to this structured part of the graph. More precisely, we build a matching $M$ inside $\bigcup_{i\in [s']} V_i$ that we call a \emph{MaMa} (short for \emph{Mandatory Matching}); the edges of $M$ will mark the places where we will attach the components of the SPICy. We also build a collection of pairs of edges $B$, the elements of which we will call  \emph{bridges}. Initially, we set $M$ and $B$ to be empty, and let $U = V(M) = \emptyset$ be the set of vertices used in the MaMa. Recall that the number of components in the SPICy is~$t\leq \ceil*{r^{1/2}}$.

Now, for each $\ell \in[t]$, consider the monochromatic path $P_{\ell} = z_0 z_1 \hdots z_{2r}$ of length $2r$ from the cycle $C_{\ell}$ in the SPICy starter. Provided  $\sum_{i\in [s']} \size{V_i\setminus U} \geq s^2+1$, there must be some part $V_k\setminus U$ of size at least $s+1$. Take a matching between some $s+1$ vertices of $V_k\setminus U$ and $z_0,z_2,\dots,z_{2s}$. By the pigeonhole principle, two of these $s+1$ matching edges, say $z_{2j}a$ and $z_{2j'}b$ for some $j<j'$ and $a, b \in V_k \setminus U$, have the same (possibly merged) colour. We then add the edge $ab$ to the MaMa, the pair $(z_{2j}a, z_{2j'}b)$ to our set of bridges, and the vertices $z_{2j+1}, z_{2j+2}, \dots, z_{2j'-2}, z_{2j'-1}$ to the set $R$ of GROVes. We further mark $a$ and $b$ as ``used'' by adding them to $U$. 

At any point in this process, the number of used vertices in $U$ is at most $2t$, so $\sum_{i\in [s']} \size{V_i\setminus U} \geq \size*{ \bigcup_{i\in [s']} V_i}  - 2t \geq s^2+1$ as needed. In each iteration, we add fewer than $2r$ vertices to the set of GROVes, so at the end we have $\size{R}\leq 2rt+ 5r \leq 2r^{3/2} + 7r \le 9r^{3/2}$. 

\paragraph{Absorbing the GROVes.}
Once we have completed the process for all the cycles in the SPICy, we wish to similarly prepare to attach the GROVes. For each GROVe $x$ in turn, suppose we still have at least $s^2 + 1$ vertices in~$\bigcup_{i\in [s']} \size{V_i\setminus U}$. We can again find a part $V_k\setminus U$ with at least $s+1$ unused vertices. Take a star from $x$ to these vertices; by pigeonhole, there are $a, b \in V_k \setminus U$ such that the edges $xa$ and $xb$ have the same colour. As before, we add the edge $ab$ to our MaMa, the pair $(xa,xb)$ to our bridges, and the vertices $a$ and $b$ to $U$. 

At each point of the process, $\size{U} \leq 2\size{R} + 2t$, so $\sum_{i\in [s']} \size{V_i\setminus U} \geq \size*{\bigcup_{i \in [s']} V_i} - 2\size{R} - 2t \geq s^2+1$, ensuring we have sufficient space to absorb all the GROVes.

\paragraph{Building the Hamilton cycle.} We now focus on the graph $\hat{G} = G[\bigcup_{i \in [s']} V_i]$. Recall that each part $V_i$ induces a monochromatic clique of colour $1$, and each bipartite graph $(V_i,V_j)$ is monochromatic. In addition, we have a MaMa~$M$ whose edges all lie within the parts. Our goal is to find a cycle passing through every vertex of $\hat{G}$ and every edge of $M$, and then use $M$ to attach the SPICy $C^{*}$ and the GROVes in $R$.

For each part $V_i$ for $i\in[s']$, we select a vertex $w_i\notin V(M)$, noting that this is possible since we only added a MaMa edge in $V_i$ if there were at least $s+1 \geq 3$ unused vertices in $V_i$. We then cover the remaining vertices with a path $Q_i$ that uses all MaMa edges in $V_i$, and label its endpoints $u_i$ and $v_i$. Consider the cycle $C' := u_1 Q_1 v_1 u_2 Q_2 v_2 \dots u_{s'} Q_{s'} v_{s'} w_1 w_2 \dots w_{s'}u_1$, which spans all vertices in $\hat{G}$; see~\cref{fig:ham-cyc} for an illustration. Note that in $C'$ every colour between parts appears an even number of times; hence, the only possible odd colour is the internal colour $1$ that is used on the edges within the paths $Q_i$.

\begin{figure}[ht]
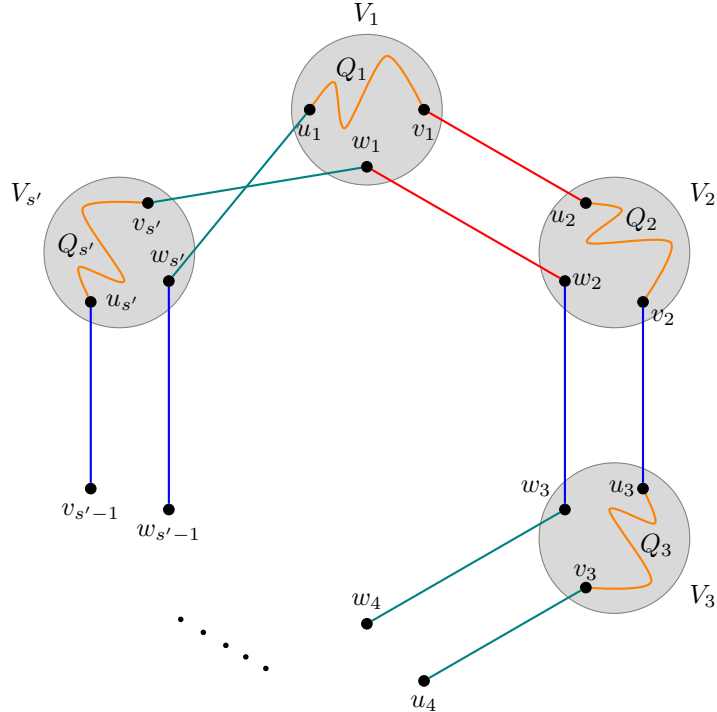

    \centering
    \HamcycleSwitchFree{0.95}
    \caption{The spanning cycle $C'$ in the switch-free portion $\hat{G}$ of $G$}
    \label{fig:ham-cyc}
\end{figure}

Now, for every edge $xy$ in the MaMa $M \subseteq C'$, we replace the edge $xy$ with the unique bridge in $B$ of the form~$(xw,yz)$ together with the structure attached to it, either a component of the SPICy $C^{*}$ or a GROVe from $R$. For each component of the SPICy, we traverse its  (non-GROVe) vertices as directed by its spanning cycle from $F^{*}$. Note that we only remove edges of colour $1$, but add an even number of edges in each colour --- indeed, the bridges came in monochromatic pairs, and $F^*$, from which we have only removed monochromatic paths of even length, was originally even-coloured. Thus, at the end, we obtain a Hamilton cycle of $K_n$, together with some embedded switches, such that all $s$ (remaining, possibly merged) colours except possibly $1$ appear an even number of times.

\paragraph{Unmerging colours.} We now process the merged colours in the reverse order in which they were merged, that is, going through the register $\cal{L}$ backwards, and using the corresponding switch to fix the parities as needed. Recall that we only merged two colours $c_1,c_2$ during the SPICy-building process when we found a $\set{c_1,c_2}$-switch and included it in the SPICy.

If a merged colour was the unique colour of odd cardinality, when we unmerge it, exactly one of the two resulting colour classes will be odd, so we do not need to do anything. If a merged colour was of even cardinality, when we unmerge it, either both colours will be even, or both colours will be odd. In the former case, we do not need to do anything, and in the latter case, we ``flip'' the switch for these two colours by traversing its flipped matching instead of its base matching, thus making both of these colours even. More precisely, suppose $S=uvxyu$ is a~$\set{c_1,c_2}$-switch embedded  in the Hamilton cycle through its base matching; let the Hamilton cycle at this step be given by $H = uvPyxQu$ for some paths $P,Q$, and assume each of $c_1$ and $c_2$ appears an odd number of times on $H$ after they are unmerged. We modify $H$ by taking the flipped matching of $S$ instead and reversing the direction of $P$, obtaining the new cycle $uyPvxQu$. Since $S$ is a $\set{c_1,c_2}$-switch, it follows that in this new cycle each of $c_1$ and $c_2$ appears an even number of times and the parities of all other colours remain unchanged. Thus, we preserve that at most one colour class is of odd cardinality.

After we have unmerged all colours in this way, we then have a Hamilton cycle in the original colouring  such that at most one colour is used an odd number of times, as required.
\end{proof}

\section{Sparser graphs}\label{sec:SparserGraphs}

In this section we investigate our problem in the setting of sparser Dirac graphs. Recall that we  defined $\rod{n}{\delta}{C_n} = \min \{ \ro{G}{C_n} : v(G) = n, \delta(G) \ge \delta \}$. 

\subsection{Upper bounds}
In this section, we prove the two upper bounds of~\cref{prop:UBSparse}. We start with the first bound, which is better when $\delta=\frac n2+O(\sqrt{n})$.

\begin{proposition}\label{prop:UB-2k}
    For every even integer $n\geq 4$  and every integer $1\leq k\leq\frac{n}{2}$, we have
    \[ \rod{n}{\frac n2 +k-1}{C_n}\leq 2k.\]
\end{proposition}

\begin{proof}
    We build an $n$-vertex graph $G$ with minimum degree $\frac n2+k-1$, together with a $2k$-edge-colouring of~$G$ where every Hamilton cycle has an odd colour class. The vertex set $V(G)$ consists of disjoint sets $A,B,C$, where $A$ and $B$ are cliques of order $\frac n2 - k$, and $C = \{c_1, c_2, \hdots, c_{2k} \}$ is a set of $2k$ vertices that are adjacent to all vertices in $G$. We clearly have $\delta(G) = \tfrac n2 + k - 1$. For the edge colouring, colour all edges from $c_i$ to $B$ with the colour $i$, and colour all other edges with the colour $1$.

    Fix a Hamilton cycle and an arbitrary vertex $a_0 \in A$, and start walking along the cycle from $a_0$. Consider the first vertex $b_0$ we visit in $B$; we must get there from some vertex $c_i \in C$. The previous edge must then have come from a vertex in $A \cup C$, and so $b_0c_i$ is the only edge of colour $i$ in the cycle, unless $i=1$.

    In that case, let $b_1$ be the last vertex we visit in $B$. We must next travel to a vertex $c_j \in C$ with $j \neq 1$. The following edge must go to a vertex in $A \cup C$, so $b_1c_j$ will be the only edge of colour $j$ in the cycle.
\end{proof}

To prove the second bound, we modify the colouring from the proof of~\cref{prop:UpperboundComplete}, exploiting the sparser nature of Dirac graphs to reduce the number of colours used.
\begin{proposition}\label{prop:UB-SparseCayley}
    For every even integer $n\geq 4$ and every real $c\in\big[\frac 12,1\big]$, we have $$r_{odd}(n, cn - 1;C_{n}) \leq \frac{3c\sqrt{2}}{2} \sqrt{n} + 2.$$
\end{proposition}

\begin{proof}
Let $m, t$ be positive integers and assume first that we have $n = m2^t$. We define a graph $G$ as the following subgraph of our construction in~\cref{prop:UpperboundComplete}. Recall that the vertex set is taken to be $\mathbb{F}_2^t \times [m]$, with colours in~$(\mathbb{F}_2^t \cup [m])\setminus\{1\}$. The vertex $(\vec{0}, m)$ is incident in $G$ to vertices of the form $(\vec{v},x)$ if and only if $x\leq \ceil{cm}$. We then fix a set~$S$ of~$\ceil{c 2^t}$ vectors in $\mathbb{F}_2^t$ with $\vec{0}\in S$. 
For all pairs $(\vec{u},x)$ and $(\vec{v},y)$ not involving $(\vec{0},m)$, we add an edge if and only if~$\vec{u} + \vec{v} \in S$. Observe that \[\delta(G)\geq \min\big\{\ceil{cm}\cdot 2^t,m\cdot\ceil*{c 2^t}\big\} - 1 \geq cm2^t - 1.\]

The colouring of $G$ is induced by our previous construction, that is, every edge between  vertices~$(\vec{0}, m)$ and~$(\vec{v}, x)$ is coloured~$x\in\ceil{cm}$ if $x>1$ and coloured $\vec{0}$ otherwise, while every remaining edge, say between vertices~$(\vec{v}_1, x_1)$ and~$(\vec{v}_2, x_2)$, is coloured~$\vec{v}_1+\vec{v}_2\in S$. It follows from the proof of~\cref{prop:UpperboundComplete} that $G$ contains no even-coloured Hamilton cycle, and hence
\[r_{odd}(m2^t, c m2^t - 1; C_{m2^t}) \leq \ceil{c 2^t} + \ceil{c m} -1\leq c(m+2^t)+1.\]

For general $n$, let $t \in \set{ \floor*{\frac12 \log_2 n}, \ceil*{\frac12 \log_2 n}}$ be the integer closest to $\frac12 \log_2 n$, and let $m = \ceil*{n/2^t}$.  Observe that~$n \le m2^t$ and
\[c(m+2^t)+1 \leq c\sqrt{n}\Big(\frac{\sqrt{n}}{2^t}+\frac{2^t}{\sqrt{n}}\Big)+2\leq \frac{3c\sqrt{2}}{2} \sqrt{n} + 2.\]
We can then embed the $n$ vertices of $G$ into the $m2^t$ vertices of this construction, ensuring that $(\vec{0}, m)$ is in the image.
\end{proof}

\cref{prop:UBSparse} now follows immediately from~\cref{prop:UB-2k,prop:UB-SparseCayley}.

\subsection{Lower bound}

As mentioned in~\cref{rem:structural}, we face additional difficulties when trying to prove lower bounds for Dirac graphs. Not only does the relative sparsity make it harder to find switches, but we also do not have the nice structural characterisation of~\cref{lem:noswitches} for switch-free colourings. In this section we show how one can consider a wider class of switches to prove the nontrivial lower bound of~\cref{thm:Dirac+const}. A tighter lower bound for $\rod{n}{\delta}{C_n}$ would likely require new ideas.

\subsubsection{Paths in super-Dirac graphs}

A graph is said to be \emph{Hamilton-connected} if any pair of vertices is connected by a Hamilton path. Every Hamilton-connected graph is also Hamiltonian, but the converse is not true. We know from Dirac's theorem~\cite{dirac1952some} that every $n$-vertex graph $G$ with $\delta(G)\geq n/2$  contains a Hamilton cycle; Ore~\cite{ore1963hamilton} proved that a slightly stronger condition on the minimum degree implies Hamilton-connectedness.

\begin{theorem}[Ore~\cite{ore1963hamilton}] \label{lem:HamiltonConnected}
Let $n\geq 3$ and $G$ be an $n$-vertex graph. If $d(u) + d(v) \geq n + 1$ for 
every  $u,v\in V(G)$ with~$uv \notin E(G)$, then $G$ is Hamilton-connected. In particular, if $\delta(G)\geq (n+1)/2$, then $G$ is Hamilton-connected.
\end{theorem}

Ore's result immediately leads to the following corollary. We note that this result requires the strictest bound on~$\delta(G)$, and so if one wishes to improve~\cref{thm:Dirac+const}, this corollary should be improved or circumvented.

\begin{corollary} \label{cor:HamiltonPath}
    Let $G$ be an $n$-vertex graph with $\delta(G) \ge \frac{n+t+1}{2}$. Then, for any $a, b \in V(G)$ and $S \subseteq V(G) \setminus \{a,b\}$ with $\card{S} \le t$, there is an $\set{a,b}$-path $P$ with $V(P) = V(G) \setminus S$.
\end{corollary}

\begin{proof}
    Consider the graph $G' = G \setminus S$. We then have $v(G') = n - \card{S}$ and $\delta(G') \ge \delta(G) - \card{S}$, and since $\card{S} \le t$, it follows that $\delta(G') \ge \frac12 ( v(G') + 1)$. Thus, by Lemma~\ref{lem:HamiltonConnected}, $G'$ has a Hamilton~$\set{a,b}$-path $P$, which is a path in $G$ with $V(P) = V(G') = V(G) \setminus S$.
\end{proof}

In addition to long paths between pairs of vertices, as in \cref{cor:HamiltonPath}, we will also need to ensure the existence of very short paths that avoid a given vertex set. This is the content of the next lemma.

\begin{lemma} \label{lem:ShortPaths}
    Let $G$ be an $n$-vertex graph with $\delta(G) \ge \frac{n+t-1}{2}$. Then, for any $a,b \in V(G)$ and~$S \subseteq V(G) \setminus \{a,b\}$ with $\card{S} \le t$, there is an $\set{a,b}$-path of length at most two that is disjoint from $S$.
\end{lemma}

\begin{proof}
    If $ab\in E(G)$, we are done, so suppose otherwise. By the degree condition, $a$ and $b$ each have at least $\frac{n+t-1}{2}$ neighbours in the $(n-2)$-element set $V(G) \setminus \{a,b\}$, and so  they must have at least $t+1$ common neighbours. At least one of these common neighbours lies outside of $S$, giving a path of length two avoiding $S$, as needed.
\end{proof}

\subsubsection{Odd-coloured cycles}

Recall that we call a cycle in a $2$-edge-coloured graph odd-coloured if it contains an odd number of edges in both colours, and even-coloured otherwise. The next lemma shows that finding a single switch is enough to build an even-coloured Hamilton cycle.  

\begin{proposition} \label{prop:C4Switch}
    For an even integer $n\geq 4$, let $G$ be an $n$-vertex graph with $\delta(G) \ge \frac{n}{2} + 2$, and suppose $E(G)$ is $2$-coloured. If there is an odd-coloured $4$-cycle, then $G$ contains an even-coloured Hamilton cycle.
\end{proposition}

\begin{proof}
    The odd-coloured $4$-cycle, say $uvxyu$, is a switch, and so we will be done if we can embed it in a Hamilton cycle. By applying~\cref{lem:ShortPaths} with $a = v$, $b = y$, and $S = \{u,x\}$, we find a $\set{v,y}$-path $Q$ of length at most two that does not contain $u$ or $x$. Next, by applying~\cref{cor:HamiltonPath} with $a = x$, $b = u$, and $S = V(Q)$, we find an $\set{x,u}$-path~$P$ with $V(P) = V(G) \setminus V(Q)$.

    Now we obtain two Hamilton cycles by traversing the base or flipped matching of the switch, namely $C_1 = uvQyxPu$ and $C_2 = uyQ vxP u$. The symmetric difference of~$C_1$ and~$C_2$ is the $4$-cycle $uvxyu$, which is odd-coloured, and hence the parities of the colour classes in $C_1$ and $C_2$ must be different. In particular, as $n$ is even, one of the two cycles is even-coloured.
\end{proof}

As mentioned earlier, in the sparse case the absence of switches does not give us enough structural information to conclude the proof, so we need another approach. The next lemma makes use of what is essentially a different type of switch, constructed from a $C_6$ rather than a $C_4$.

\begin{proposition} \label{prop:C6Switch}
    Let $G$ be an $n$-vertex graph with $\delta(G) \ge \frac{n}{2} + 4$, and suppose $E(G)$ is $2$-coloured. If there is an odd-coloured $6$-cycle, then $G$ contains an even-coloured Hamilton cycle.
\end{proposition}

\begin{proof}
    Let the odd-coloured $6$-cycle be $uvwxyzu$. Applying~\cref{lem:ShortPaths} with $a = v$, $b = z$, and $S = \{u, w, x, y\}$, we find a $\set{v,z}$-path $Q_1$ of length at most two that is disjoint from the rest of the $6$-cycle. Applying the lemma once again, this time with $a = y$, $b = w$, and $S = \{u, x\} \cup V(Q_1)$, we find a $\set{y,w}$-path $Q_2$  of length at most two that is disjoint from the rest of the $6$-cycle and the path $Q_1$. Finally, we can apply~\cref{cor:HamiltonPath} with $a = x$, $b = u$, and~$S = V(Q_1) \cup V(Q_2)$ to find an $\set{x,u}$-path $P$ that covers all the vertices in the graph except those on the paths~$Q_1$ and~$Q_2$.

    Now consider the Hamilton cycles $C_1 = uvQ_1zyQ_2wxPu$ and $C_2 = uzQ_1v wQ_2yxP u$. The symmetric difference of these paths is the $6$-cycle $uvwxyzu$, which is odd-coloured, and thus one of the two cycles must be even-coloured.
\end{proof}

\subsubsection{Odd-Ramsey numbers in super-Dirac graphs} 

Armed with these preliminaries, we can now show that $n$-vertex graphs with minimum degree~$\frac{n}{2} + 4$ require at least three colours in any odd-Ramsey colouring.

\begin{proof}[Proof of~\cref{thm:Dirac+const}]
    Suppose for a contradiction that we have an $n$-vertex graph $G$ with $\delta(G) \ge \frac{n}{2} + 4$, together with a $2$-edge-colouring in which every Hamilton cycle is odd-coloured.

    Consider two vertices $x, y \in V(G)$.  If there is one common neighbour~$u$ that sends edges of opposite colours to~$x$ and~$y$, and one common neighbour~$v$ that sends edges of the same colour to~$x$ and~$y$, then $uxvyu$ forms an odd-coloured $4$-cycle. By~\cref{prop:C4Switch},~$G$ contains an even-coloured Hamilton cycle, a contradiction. It follows that either all common neighbours send edges of the same colour to $x$ and $y$, in which case we say $x$ and $y$ \emph{agree}, or all common neighbours send edges of opposite colours to $x$ and $y$, in which case we say $x$ and $y$ \emph{disagree}.

    We now argue that agreeing defines an equivalence relation on $V(G)$. Suppose this were not the case. Then we would have vertices~$x$,~$y$, and~$z$, such that~$x$ and~$y$ agree, and~$y$ and~$z$ agree, but~$x$ and~$z$ disagree. By the minimum degree condition, each pair has at least four common neighbours; thus, we can find common neighbours~$u$ of~$x$ and~$y$,~$v$ of~$y$ and~$z$, and~$w$ of~$x$ and~$z$ such that~$u,v$, and~$w$ are distinct and different from~$x,y$, and~$z$.
    Then $xuyvzwx$ forms a $6$-cycle. Since~$x$ and~$y$ agree, the colours of~$xu$ and~$uy$ are the same, and this is similarly true of the pair of edges~$yv$ and~$vz$. However, since~$x$ and~$z$ disagree, the edges~$zw$ and~$wx$ have differing colours, and thus this is an odd-coloured $6$-cycle. By~\cref{prop:C6Switch}, we would then find an even Hamilton cycle.

    Thus, the pairs of agreeing vertices form an equivalence relation. Suppose we had at least three equivalence classes, and let $x$, $y$, and $z$ come from different classes, so that each pair disagrees. As above, we can find distinct vertices $u$, $v$, and $w$ such that $xuyvzwx$ forms a $6$-cycle. Since $x$, $y$, and $z$ pairwise disagree, it follows that  the set $\set{xu, uy, yv, vz, zw, wx}$ contains precisely three edges of each colour, and then we once again find an even Hamilton cycle, courtesy of~\cref{prop:C6Switch}.

    This implies that the vertices $V(G)$ can be partitioned into two equivalence classes $A$ and $B$ (one of which may be empty), such that all pairs within $A$ and within $B$ agree, while all pairs between $A$ and $B$ disagree.

    Finally, given any Hamilton cycle $C = v_1 v_2 \hdots v_nv_1$ in $G$, we count the number of edges in $C$ of each colour by considering  the vertices in even positions, that is, $v_2, v_4, \hdots, v_n$. The cyclic sequence $v_2, v_4, \dots, v_n$ must cross between the classes $A$ and $B$ an even number of times; therefore an even number of pairs $v_{2i} v_{2i+2}$ (addition modulo $n$) disagree, while the remaining pairs agree. But every agreeing pair contributes an even number of red edges and an even number of blue edges in $C$, while every disagreeing pair contributes one edge of each colour. As there is an even number of such pairs, it follows that $C$ is even-coloured, which is a contradiction.
\end{proof}

\section{Concluding remarks}\label{sec:ConcludingRemarks}

In this paper, we studied odd-Ramsey numbers of Hamilton cycles, both when the host graph is complete and in the sparser setting of (super-)Dirac graphs. We resolved the former problem up to a constant factor, and gave some initial results in the latter setting. A number of possible further directions arise from our work.

\paragraph{Closing the gap in \cref{thm:complete_graphs}} We have determined the odd-Ramsey number of Hamilton cycles up to a factor of $3$, and it would be very interesting to see which of our bounds, if either, is closer to the truth.

There are two main bottlenecks when it comes to improving the constant in the lower bound of~\cref{thm:complete_graphs}. The first is the construction of the SPICy starter, where we use~\cref{lem:longevencycles} to find long monochromatic cycles of even lengths. While~\cref{lem:longevencycles} is tight, we do not actually require the cycles to be monochromatic --- it suffices for the union of the cycles to be even-coloured, and for the path of length $2r$ on each cycle to be built out of monochromatic cherries. The second issue is that we potentially use $(2+o(1))r^2$ vertices when finding and attaching switches to the SPICy, and need to ensure we have sufficiently many vertices remaining for the rest of the argument. Thus, one could potentially narrow the gap by finding more efficient means of building a SPICy starter, and then embedding switches in it.

On the other hand, the upper bound comes from a rather surprising construction that modifies a Cayley colouring over a finite field at a single vertex, using (roughly) equal-sized sets of colours for both parts. We would like to live in a world where this colouring is optimal, and thus pose the potentially provocative question below.

\begin{question}\label{conj:BetterLower}
Is it true that, for every $\varepsilon>0$ and every sufficiently large even integer $n$, we have
\[\ro{n}{C_n}\geq (2-\varepsilon) \sqrt{n}\ ?\]
\end{question}
In conjunction with the upper bound of~\cref{prop:UpperboundComplete} for $n = 4^t$, a positive answer to~\cref{conj:BetterLower} would imply \[\liminf_{k\to\infty}\frac{r_{odd}(2k,C_{2k})}{\sqrt{2k}}=2.\]

\paragraph{Hamilton paths}

In this paper we focused on odd-Ramsey numbers of Hamilton cycles, but the analogous question for Hamilton paths when the number of vertices is odd is equally interesting. Note that our lower bound still holds; \cref{thm:SqrtLowerBound} shows that, in any edge-colouring of $K_n$ with fewer than $\parens*{\frac{\sqrt{2}}{2} + o(1)}\sqrt{n}$ colours, there is a Hamilton cycle with one odd colour class. Removing an edge from said colour class then provides an even-coloured Hamilton path, showing $\ro{n}{P_n} = \Omega(\sqrt{n})$. We believe that the behaviour of Hamilton paths and cycles should not be very different, but we have not been able to find a matching construction to establish the upper bound.

\begin{question}\label{question:Path}
    Is it true that $\ro{n}{P_n} = O(\sqrt{n})$?
\end{question}

\paragraph{Sparse odd-Ramsey}
In \cref{sec:SparserGraphs} we made some initial progress towards understanding odd-Ramsey numbers when the host graph is not complete. 
For sparse host graphs, we know from~\cref{prop:UB-2k} that ${\rod{n}{\frac n2 +k-1}{C_n}\leq 2k}$ for every integer $k\leq\frac{n}{2}$, while~\cref{prop:UB-SparseCayley} shows $\rod{n}{cn-1}{C_n} \le \frac{3c\sqrt{2}}{2}\sqrt{n} + 2$ for all $c \in [\frac12, 1]$. Meanwhile,~\cref{thm:Dirac+const} shows $\rod{n}{\frac n2+4}{C_n}\geq 3$. Many questions remain open; we highlight three of them. 

A first natural problem is to understand the true behaviour of $\rod{n}{\frac n2 + k}{C_n}$ as a function of $k$. In particular, our upper bound grows linearly in $k$ with slope $2$ for $k = O(\sqrt{n})$, and with slope $\frac{3}{\sqrt{2n}}$ for larger values of $k$. It would be interesting to prove better upper bounds for these ranges of $k$.

As a second problem, we wonder how $\rod{n}{n-t}{C_n}$ differs from $\ro{n}{C_n}$ when $t$ is small. As a first open case, we propose studying whether or not $\ro{K_n}{C_n}$ is different from $\ro{K_n-M}{C_n}$, where $M$ is a perfect matching. 

Finally, we know that $\rod{n}{\frac n2}{C_n}=2$ and that $\rod{n}{\frac n2 + 4}{C_n}\geq 3$. We tend to believe that the transition from two to three happens already at $\frac n2+1$, but we have been unable to show it. To reduce the minimum degree required for the lower bound, one would need to improve our use of Ore's result guaranteeing Hamilton-connectivity.

In general, proving lower bounds in this setting seems considerably more difficult than when the host graph is complete. One reason is that~\cref{lem:noswitches} provides a nice structural result for switch-free colourings of complete graphs, but there is no natural analogue of this result in the sparse setting.

In a very different direction, it would be interesting to understand the behaviour of odd-Ramsey numbers with respect to Hamilton cycles when the host graph is random. 
\begin{question}
  What is $\ro{G}{C_n}$ when $G\sim G(n,p)$ for different values of $p$? What is the threshold for~$\ro{G}{C_n}>2$?
\end{question}

Note that, while a random graph $G$ is almost surely Hamiltonian as soon as its minimum degree is at least two, this is not enough to ensure $\ro{G}{C_n} > 2$. Indeed, given a vertex $v$ of degree two, colour one of its incident edges blue, and all other edges red. Then every Hamilton cycle must contain this blue edge, and thus has an odd colour class. Given the thresholds for $G(n,p)$ to have minimum degree two or three, see e.g.~\cite{janson2011random}, this implies that, when $\frac{\log n + \log \log n}{n} < p < \frac{\log n + 2 \log \log n}{n}$, the random graph $G \sim G(n,p)$ will almost surely be Hamiltonian, but have~$\ro{G}{C_n} = 2$.

\paragraph{Unique colourings}

Here we discuss one more variant of the \Erdos-\Gyarfas\  problem, concerned with unique-chromatic colourings, that we believe is worth investigating.
The following definition appears in~\cite{yip2024k8}, but has been considered earlier by Radoicic and by Axenovich and Conlon in unpublished work.

\begin{definition}
    Given an edge-colouring of $K_n$, a copy of $H$ in $K_n$ is \emph{unique-chromatic} if it meets some colour class in precisely one edge. An edge-colouring of $K_n$ is \emph{$H$-unique} if all copies of $H$ in $K_n$ are unique-chromatic, and we write $r_u(n,H)$ for the smallest number of colours needed in an $H$-unique colouring of $K_n$.
\end{definition}

Trivially, $\ro{n}{H}\leq r_u(n,H)$ for any choice of $n$ and $H$, and one advantage of $r_u(n,H)$ is that the problem is interesting irrespective of the parity of the number of edges of $H$.
Some first results regarding this parameter appear in~\cite{BDLP2024OddBip,yip2024k8}. In light of the topic of this paper, we propose investigating $r_u(n,C_n)$. 

\begin{question}
    Determine $r_u(n,C_n)$. How large is the ratio $\frac{r_u(n,C_n)}{\ro{n}{C_n}}$?
\end{question}
In~\cite{BDLP2024OddBip}, we show that the two parameters can be a constant factor apart for the family of all cliques, but we wonder if they differ by much more for Hamilton cycles. We tend to think that $r_u(n,C_n)$ should be linear in $n$, which, if true, would show a large separation between the two parameters.

\paragraph{Acknowledgements.}
We thank Anurag Bishnoi and Carla Groenland for fruitful discussions, and especially for bringing~\cite{li2023eigenvalues} to our attention and for raising~\cref{question:Path}.

SB: Parts of this research were conducted while the author was at the School of Mathematics, University of Birmingham, Birmingham, United Kingdom. The research leading to these results was supported by EPSRC, grant no.\ EP/V048287/1 and by ERC Advanced Grants ``GeoScape'', no.\ 882971 and ``ERMiD'', no.\ 101054936. There are no additional data beyond that contained within the main manuscript.

SD: Research supported by Taiwan NSTC grant 113-2628-M-002-008-MY4.

KP: This project has received funding from the European Union’s Horizon 2020 research and innovation programme
under the Marie Sk{\l}odowska-Curie grant agreement No 101034413.\includegraphics[width=5.5mm, height=4mm]{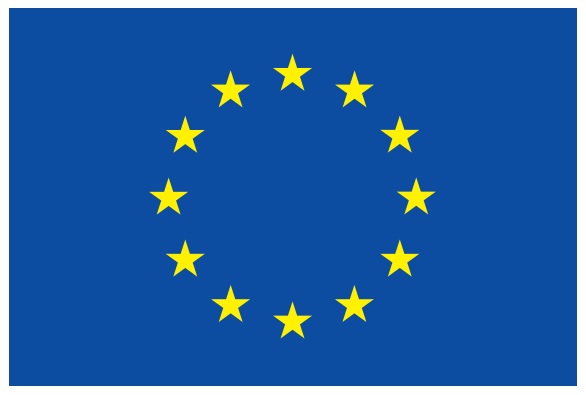}

\bibliographystyle{plain}
\bibliography{bibliography.bib}

\end{document}